\documentclass[11pt]{article}
\usepackage{amsthm}
\usepackage{url}
\usepackage[margin=2.5cm,nohead]{geometry}
\newtheorem{theorem}{Theorem}
\newtheorem{lemma}{Lemma}
\newtheorem{corollary}{Corollary}
\newtheorem{proposition}{Proposition}

\newtheorem{definition}{Definition}

\usepackage{amssymb,amsfonts,amstext,amsmath}
\usepackage{xcolor}
\usepackage{cite}

\DeclareMathOperator{\grad}{grad}  
\usepackage{amssymb}

\DeclareMathOperator{\Hess}{Hess}

\usepackage{geometry}
\usepackage{color}
\usepackage{float}
\usepackage{textcomp}
\usepackage{amsthm}
\usepackage{amsmath}
\usepackage{amssymb}

\usepackage{changes}

\usepackage{mathtools}

\begin{document}

\title{Iteration-complexity  of gradient,  subgradient  and proximal point  methods on Riemannian manifolds \thanks{This work was supported by CNPq.}}

\author{
G. C. Bento \thanks{IME/UFG, Avenida Esperan\c{c}a, s/n, Campus Samambaia,  Goi\^ania, GO, 74690-900, Brazil (e-mails: {\tt glaydston@ufg.br}, {\tt orizon@ufg.br}, {\tt jefferson@ufg.br}).}
\and
 O. P. Ferreira\footnotemark[2]''
\and
J. G. Melo \footnotemark[2]
}

\maketitle

\begin{abstract}
This paper considers optimization problems on Riemannian manifolds and analyzes iteration-complexity for gradient and subgradient methods on manifolds with non-negative curvature. By using tools from the Riemannian convex analysis  and exploring directly the tangent space of the manifold,  we obtain different iteration-complexity bounds for the aforementioned methods, complementing and improving  related results. Moreover,  we also establish  iteration-complexity bound for the proximal point method on Hadamard manifolds. \\

\noindent
{\bf keywords:} Complexity,  gradient method, subgradient method,  proximal point method,  Riemannian manifold.

\end{abstract}
\section{Introduction}
Optimization methods  in the Riemannian setting  have been the subject of intense research; see, for example, \cite{ZhangSra2016, BoumalAbsilCartis2016, GrohsHosseini2016, WangLiYao2015, BentoFerreiraOliveira2015, WangLiWangYao2015,  LiMordukhovichWang2011,  LiYao2012,    NesterovTodd2002,   Smith1994, Manton2015}.  One advantage of this study is the possibility to transform some Euclidean non-convex problems  into  Riemannian convex  problems, by introducing a suitable metric, and thus, enabling the modification of  numerical  methods for the purpose of finding a global minimizer;  see \cite{BentoMelo2012, BentoFerreiraOliveira2015, FerreiraCPN2006, Rapcsak1997} and references therein. Furthermore, many optimization problems are naturally posed on Riemannian manifolds which have a specific underlying geometric and algebraic structure that can be exploited to greatly reduce the cost of obtaining solutions. For instance, in order to take advantage of the Riemannian geometric  structure, it is preferable to treat certain constrained optimization problems as problems for finding singularities of gradient vector fields  on  Riemannian manifolds rather than using Lagrange multipliers  or projection methods; see \cite{Luenberger1972,    Smith1994, Udriste1994}. Accordingly, constrained optimization problems are viewed as unconstrained ones from a Riemannian geometry point of view. Besides, Riemannian geometry also opens up new research directions  that aid in developing competitive methods;  see \cite{EdelmanAriasSmith1999, NesterovTodd2002, Smith1994}.


The gradient method is one of the oldest  optimization methods considered in the Riemannian  context. As far as we know,  the early works dealing with this method  include  \cite{Luenberger1972, Gabay1982, Udriste1994, Smith1994, Rapcsak1997, daCruzNetoLimaOliveira1998}.  In order to deal with non-smooth convex optimization problems on Riemanian manifolds,  \cite{FerreiraOliveira1998} proposed and analyzed a subgradient method which is quite simple and possess nice convergence properties. Since  then, the subgradient method in the Riemannian setting has been studied in different context; see, for instance, \cite{BentoMelo2012, WangLiWangYao2015, WangLiYao2015, GrohsHosseini2016}. One of the most interesting optimization methods is the proximal point method which was  first proposed in the linear context by  \cite{Martinet1970} and extensively studied by \cite{Rockafellar1976}.   In the  Riemannian setting, the proximal point method was first  studied in \cite{FerreiraOliveira2002} for   convex optimization problems on Hadamard manifold and has been extensively explored since then; see, for example,   \cite{LiLopesMartin-Marquez2009,  BentoFerreiraOliveira2015, BentoNetoOliveira2016,  SouzaOliveira2015}  and references therein. 
The asymptotic convergence analyses of optimization methods in the Riemannian setting have been analyzed by many papers (see, for example, \cite{Gabay1982, Luenberger1972,  Udriste1994,  Smith1994,  Rapcsak1997,  daCruzNetoLimaOliveira1998, PapaQuispeOliveira2008, FerreiraOliveira1998, FerreiraOliveira2002}),  however, only a few number of papers has studied iteration-complexity in the Riemannian context; see\cite{ZhangSra2016, BoumalAbsilCartis2016, ZhangReddiSra2016}.   In \cite{ZhangSra2016}, the authors considered  convex optimization problems on Hadamard manifolds and  obtained iteration-complexity bounds for some variants of  gradient  and subgradient methods.  In \cite{BoumalAbsilCartis2016}, the authors established some iteration-complexity bounds  for  gradient method  and trust region method on  Riemannian manifold without any assumption on its curvature or convexity of the  problem.  In \cite{ZhangReddiSra2016},   the authors presented  a fast stochastic Riemannian method   for solving structured optimization problems   as well as  some bounds for its  iteration-complexity.  From the above discussion, we see that  iteration-complexity analysis of optimization methods on Riemannian manifolds is  an interesting research subject.

In this paper, we analyze iteration-complexity of gradient, subgradient and proximal point methods in the Riemannian setting.    By using tools from the Riemannian convex analysis  and exploring directly the tangent space of the manifold,  we obtain different iteration-complexity bounds for the gradient and  subgradient methods on manifolds with non-negative curvature, complementing and improving  related results; see \cite{ZhangSra2016, BoumalAbsilCartis2016}.   More specifically, in comparison to \cite{ZhangSra2016}, we overcome some of its technical difficulties which obliged the authors to study the gradient and subgradient methods on Hadamard manifolds. In contrast to \cite{BoumalAbsilCartis2016}, we make use of convexity in the Riemannian context, allowing us to  improve some of their iteration-complexity bounds for the gradient method.  Besides, we establish  iteration-complexity bound for the proximal point method on Hadamard manifolds under convexity assumption on the objective function. As far as we know, this paper is the first one to present iteration-complexity bound for the proximal point method in the Riemannian setting. 

This paper is organized as follows. Section~\ref{sec:aux}  presents some definitions and auxiliary results   related to the Riemannian geometry  that are important to  our study. Our  optimization problem  is  stated at the end of this section.  In Section~\ref{sec:gradient}, we review the gradient method and presents its teration-complexity analysis. In Section~\ref{sec:subgradient}, we consider non-smooth convex optimization problems and analyzes the subgradient method. Section~\ref{sec:proximalpoint}  is devoted to the iteration-complexity analysis of the  proximal point method. The last section contains a conclusion.
\section{Notations and basic results} \label{sec:aux}
In this section, we recall  some  concepts, notations and basics results  about Riemannian manifolds.   For more details   see, for example, \cite{doCarmo1992, Sakai1996, Udriste1994,  Rapcsak1997}.

 We denote by $T_pM$ the {\it tangent space} of a Riemannian manifold $M$ at $p$.  The corresponding norm associated to the Riemannian metric $\langle \cdot ~, ~ \cdot \rangle$ is denoted by $\|  ~\cdot~ \|$. We use $\ell(\gamma)$ to denote the length of a piecewise smooth curve $\gamma:[a,b]\rightarrow M$. The Riemannian  distance  between $p$ and $q$   in a finite dimensional Riemannian manifold $M$ is denoted  by $d(p,q)$,  which induces the original topology on $M$, namely,  $( M, d)$ is a complete metric space and bounded and closed subsets are compact. 
  Let $( N,  \langle \cdot ~, ~ \cdot \rangle)$ and $(M,   \langle\!\!\langle \cdot ~, ~ \cdot \rangle\!\!\rangle)$ be Riemannian manifolds and $\Phi : N \to M$ be an isometry, that is, $\Phi$  is $C^{\infty}$, and for all $q \in N$ and $u, v \in T_qN$, we have $\langle u , v \rangle=\langle\!\!\langle d\Phi_q u , d\Phi_q v \rangle\!\!\rangle $, where $d\Phi_q : T_{q}N \to T_{\Phi(q)}M$ is the differential of $\Phi$ at $q\in N$.  One can verify that $\Phi$  preserves geodesics, that is, $\beta$ is a geodesic in $N$ if only if $\gamma=\Phi\circ \beta$ is a geodesic in $M$.   Denote by ${\cal X}(M)$ the space of smooth vector fields on $M$. Let $\nabla$ be the Levi-Civita connection associated to $(M, \langle \cdot ~, ~ \cdot \rangle)$. The Riemannian metric induces a mapping  $f\mapsto\grad f $ which associates to each  real differentiable function  over $M$ its {\it gradient} via the rule $\langle\grad f,X\rangle=d f(X),\ X\in{\cal X}(M)$ and  a  mapping  $f\mapsto \Hess f$  which associates to each twice differentiable function its {\it hessian}  via the rule     $\langle\Hess f X,X\rangle=d^2 f(X, X),\ X\in{\cal X}(M)$, where $ \Hess f X:= \nabla_{X} \grad f$.    The {\it norm of a  linear map} $A:T_p M \to T_p M$  is defined by $\|A\|:=\sup \left\{ \|Av \|~:~ v\in T_p M, \,\|v\|=1 \right\}$.  A vector field $V$ along $\gamma$ is said to be {\it parallel} iff $\nabla_{\gamma^{\prime}} V=0$. If $\gamma^{\prime}$ itself is parallel we say that $\gamma$ is a {\it geodesic}. Given that geodesic equation $\nabla_{\ \gamma^{\prime}} \gamma^{\prime}=0$ is a second order nonlinear ordinary differential equation, then geodesic $\gamma=\gamma _{v}( \cdot ,p)$ is determined by its position $p$ and velocity $v$ at $p$. It is easy to check that $\|\gamma ^{\prime}\|$ is constant. The restriction of a geodesic to a  closed bounded interval is called a {\it geodesic segment}. A geodesic segment joining $p$ to $q$ in $ M$ is said to be {\it minimal} if its length is equal to $d(p,q)$. For each $t \in [a,b]$, $\nabla$ induces an isometry, relative to $ \langle \cdot , \cdot \rangle  $, $P_{\gamma,a,t} \colon T _{\gamma(a)} {M} \to T _
{\gamma(t)} {M}$ defined by $ P_{\gamma,a,t}\, v = V(t)$, where $V$ is the unique vector field on $\gamma$ such that
$ \nabla_{\gamma'(t)}V(t) = 0$ and $V(a)=v$, the so-called {\it parallel transport} along  the geodesic segment   $\gamma$ joining  $\gamma(a)$ to $\gamma(t)$.  When there is no confusion we will  consider the notation $P_{\gamma,p,q}$  for the parallel transport along  the geodesic segment  $\gamma$ joining  $p$ to $q$.  A Riemannian manifold is {\it complete} if the geodesics are defined for any values of $t\in \mathbb{R}$. Hopf-Rinow's theorem asserts that any pair of points in a  complete Riemannian  manifold $M$ can be joined by a (not necessarily unique) minimal geodesic segment.  Due to  the completeness of the Riemannian manifold $M$, the {\it exponential map} $\exp_{p}:T_{p}  M \to M $ can be  given by $\exp_{p}v\,=\, \gamma _{v}(1,p)$, for each $p\in M$. A complete simply connected Riemannian manifold of non-positive sectional curvature is called a {\it Hadamard manifold}.  For all $p\in M$, the  exponential map $exp_{p}:T_{p}  M \to M $  is  a diffeomorphism and $exp^{-1}_{p}:M\to T_{p}M$ denotes its inverse. In this case, $d({q}\, , \, p)\,=\,||exp^{-1}_{p}q||$ and   the function $d_{q}^2: M\to\mathbb{R}$ defined by $ d_{q}^2(p):=d^2(q,p)$ is  $C^{\infty}$ and $ \grad d_{q}^2(p):=-2exp^{-1}_{p}{q}.$

 {\it In this paper, all manifolds are assumed to be  connected,   finite dimensional and complete.}
 
\begin{proposition} \label{pr:lawcosines}
Let $\gamma _{1} $ and $\gamma _{2}$ be   geodesic segments such that $\gamma_{1}(0)=\gamma_{2}(0)$ and $\gamma_{1}$ be  minimal. Then, letting  $\ell_{1}=\ell(\gamma _{1})$, $ \ell_{2}=\ell(\gamma _{2})$, $\ell_{3}=d(\gamma_{1}(\ell_{1}),\gamma _{2}(\ell_{2}))$ and $\alpha $ be the angle between $\gamma _{1}^{\prime}(0)$ and $\gamma _{2}^{\prime}(0)$,    the following statements  hold:
\begin{itemize}
\item[(i)] 
If  $M$  has   non-negative curvature, then   $\ell_{3}^{\,2}\leq \ell_{1}^{\,2}+\ell_{2}^{\,2}- 2 \ell_{1} \ell_{2} \cos \alpha$. Consequently,  for each $p\in M$ and $u, v\in T_{p}M$, there holds 
$
d(\exp_{p}u, \exp_{p}w)\leq \|u-v\|.
$
\item[(ii)] If  $M$  has   non-positive curvature, then $ \ell_{3}^{\,2}\geq \ell_{1}^{\,2}+\ell_{2}^{\,2}- 2 \ell_{1} \ell_{2} \cos \alpha . $ Consequently,  for each $p\in M$ and $u, v\in T_{p}M$, there holds
$
d(\exp_{p}u, \exp_{p}w)\geq \|u-v\|.
$
\end{itemize}
\end{proposition}
Now, we recall some concepts and basic properties  about  convexity   in the  Riemannin context and the concept  of Lipschitz continuity of functions.    A set,  $\Omega\subseteq M$ is said to be {\it convex}  iff any geodesic segment with end points in $\Omega$ is contained in $\Omega$.  A function $f:M\to\mathbb{R}$  is said to be  {\it convex} on a convex set $\Omega $ iff for any geodesic segment $\gamma:[a, b]\to\Omega$ the composition $f\circ\gamma:[a, b]\to\mathbb{R}$ is convex.  A vector $s \in T_pM$ is said to be a {\it subgradient\/} of the function  $f$ at $p$, iff 
\begin{equation} \label{eq:defSugrad}
f(\exp_p v) \geq f(p) + \left\langle s, \, v \right\rangle,  \quad v\in T_pM.
\end{equation}
Let  $\partial f(p)$ be  the {\it subdifferential\/} of $f$ at $p$, namely, the set of all subgradients of $f$ at $p$. Then,   $f$ is convex iff there holds
\begin{equation}\label{eq:CharConvexFunc}
f(\exp_p v) \geq f(p) + \left\langle s, \, v \right\rangle,    \qquad p\in M, \quad s\in \partial f(p), \quad v\in T_pM.
\end{equation}
If  $f:M\to\mathbb{R}$ is a  differentiable function, then  $\partial f(p)=\{\grad f(p)\}$  and we have the characterization: the function  $f$ is convex iff there holds
\begin{equation}\label{eq:CharDifConvexFunc}
f(\exp_p v)\geq f(p)+\langle \grad f(p), v\rangle, \qquad p\in M, \quad v\in T_pM.
\end{equation}
 \begin{definition} \label{Def:FunctionLips}
 A function $f:M \to \mathbb{R}$ is  said  to be   Lipschitz continuous with  constant $\tau \geq 0$ if,   for any  points $ p$ and $q\in M$,   it  holds that $ \left |f(p)- f(q)\right |\leq \tau\, d(p,q).$
\end{definition}
Next we define the concept of Lipschitz continuity of gradient vector fields (see  \cite{daCruzNetoLimaOliveira1998}) and present some basic properties related to this concept.
\begin{definition} \label{Def:GradLips}
 Let $f:M \to \mathbb{R}$ be a  differentiable  function. The gradient vector  field   of $f$  is  said  to be   Lipschitz continuous with  constant $L\geq 0$ if,   for any  points $ p$ and $q\in M$ and   $\gamma$  a  geodesic  segment joining $p$ to $q$, it holds that
$\left\|P_{\gamma,p,q} \grad f(p)- \grad f(q)\right\|\leq L d(p, q).$
\end{definition}
\begin{lemma} \label{le:lc}
Let $f:M \to \mathbb{R}$ be a  differentiable  function such that its  gradient vector  field   is   Lipschitz continuous with  constant $L\geq 0$. Then, 
\[
f( \exp_{p}(v)) \leq f(p) + \langle\grad f(p), v\rangle+\frac{L}{2}\left\|v\right\|^{2}, \qquad  p\in M, \quad v\in T_pM .
\]
\end{lemma}
\begin{proof}
Let     $p\in M$ and $v\in T_pM$ and   $\gamma(t):=\exp_{p}(t v)$, for $t\in  \mathbb{R}$. Note that $\gamma(0)=p$ and $\gamma'(t)= P_{\gamma, p,\gamma(t)}v$. Thus,  we have
$$
f( \exp_{p}(v))= f(p) + \int_0^1 \left\langle\grad f(\gamma(t)), P_{\gamma, p, \gamma(t)}v \right\rangle dt.
$$
Considering that  the  parallel transport is  an isometry,   after some manipulations in the last equality we obtain 
$$
f( \exp_{p}(v))= f(p) +  \langle\grad f(p),v\rangle+  \int_0^1\left\langle[ \grad f(\gamma(t))-P_{\gamma, p, \gamma(t)}\grad f(p)],   P_{\gamma, p, \gamma(t)}v\right\rangle dt.
$$
Using Cauchy-Schwarz inequality,  that  $\grad f$   is   Lipschitz continuous with  constant $L$, $d(p, \gamma(t))=t\|v\|$,     $\gamma'(t)= P_{\gamma, p,\gamma(t)}v$  and  the isometry of the parallel transport, it follows from the last equality that
$$
f( \exp_{p}(v))\leq f(p) +  \langle\grad f(p),v\rangle+  L  \|v\|^2    \int_0^1  t ds, 
$$
which after  performing the integral gives the desired result.
\end{proof}
Next result estimates the decrease of a function $f$ along the negative direction of its gradient vector field. This is a key result to provide iteration-complexity bounds for the gradient method  on a general Riemannian manifold.
\begin{corollary} \label{cr:cl}
 Let $f:M \to \mathbb{R}$ be a  differentiable  function with an L-Lipschitz continuous gradient vector  field. Then, 
\[
f( \exp_{p}(- t \grad f(p))) \leq f(p)- \left(t-\frac{L}{2}t^{2}\right)\left\|\grad f(p)\right\|^{2}, \qquad t\in \mathbb{R}, \quad p\in M.
\]
\end{corollary}
\begin{proof}
The proof follows directly from Lemma~\ref{le:lc} by taking $v=- t \grad f(p)$.
\end{proof}
In this paper, we are interested in the following optimization problem
\begin{equation} \label{eq:OptP}
\min \{ f(p) ~:~   p\in M\},
\end{equation}
where $M$ is  a  Riemannian manifold  and $f:M \to \mathbb{R}$ is a  continuously differentiable  and/or convex function. {\it From now on, we assume that the solution set of the problem in  \eqref{eq:OptP} is non-empty and denote  its  optimum value by $f^*$}.

\section{Iteration-complexity analysis}
This section is divided into three  subsections. The first one presents some iteration-complexity bounds for the gradient method while the second one analyzes complexity bounds for the subgradient method.   Our main results in this  subsections   assume convexity of the objective function and that the   Riemannian Manifold has non-negative curvature. The third subsection  is devoted to the iteration-complexity analysis of the  proximal point method under convexity  of the objective function  on Hadamard manifolds.
\subsection{Gradient method}\label{sec:gradient}
In this subsection, we recall the gradient method for solving problem \eqref{eq:OptP} and present three results which analyze  iteration-complexity of the method. We first consider the method in a general Riemannian manifold and recover the $\mathcal{O} (1/\varepsilon^2)$ worst-case complexity bound to obtain $p_N\in M$ satisfying  $\|\grad f(p_N)\|< \varepsilon$, where $\varepsilon$ is a given tolerance.  The subsequent two results restrict the sign of the curvature to be non-negative and assume convexity of the objective function. Under these assumptions, we show that the worst-case iteration-complexity bound  $\mathcal{O} (1/\varepsilon^2)$,  obtained for the general case, can be improved to $\mathcal{O} (1/\varepsilon)$.

\noindent
In the following,  we formally state the gradient method to solve \eqref{eq:OptP},  where the objective function is assumed to be continuously differentiable.
\\[3mm]
\noindent
\fbox{
\begin{minipage}[h]{5.55 in}
{\bf Gradient Method}
\begin{itemize}
\item[(0)] Let an initial point $p_0 \in M$, and set $k=0$;
\item[(1)] choose a stepsize $t_k>0$ and computes
\begin{equation} \label{eq:GradMethod}
p_{k+1}:=\mbox{exp} _{p_{k}}\left(-t_{k}\,\mbox{grad} f(p_{k})\right);
\end{equation}
\item[(2)] set $k\leftarrow k+1$ and go to step 1.
\end{itemize}
\noindent
\end{minipage}
}
\\[3mm]

\noindent
This method is a natural extension of the classical gradient method to the Riemannian setting. It has been  extensively studied in different contexts; see, for example, \cite{Udriste1994, Smith1994, Rapcsak1997, daCruzNetoLimaOliveira1998, PapaQuispeOliveira2008}. Similarly to the classical gradient method, the stepsize $t_k$ can be chosen by an Armijo line search or, depending on the structure of the problem \eqref{eq:OptP}, by some exogenous procedure such as $\sum_k t_k=\infty$ and $\sum_k t_k^2<\infty$, guaranteeing that the stepsizes are not too small and not too large. It is interesting to note that, for objective functions with Lipschitz continuous gradient, the analysis of the gradient method with an Armijo line search it is quite similar to the case where constant  stepsizes are considered, so, for the sake of simplicity, this will be the case  in this subsection. The exogenous rule will be considered only in the analysis of the subgradient method in the next subsection which does not assume differentiability of the objective function.
 
 \noindent In the following,  we present an  iteration-complexity bound related to the gradient method on a general Riemannian manifold. This result has already appeared  in \cite{BoumalAbsilCartis2016}, but, since its proof is very simple and short, we consider it for the sake of completeness. 
\begin{theorem}\label{th:grad1} Let  $\{p_k\}$ be the  sequence  generated by the gradient method with constant stepsizes $t_{k}=1/L$, for all $k\geq 0$. Then,   for every $N\in \mathbb{N}$,  there  holds
$$
\min \left\{\|\grad f(p_{k})\|~;~ k=0, 1,\ldots, N \right\}\leq \frac{\sqrt{ 2L(f(p_0) -f^*)}}{\sqrt{N+1}}.  
$$
As a consequence, given  a tolerance $\epsilon>0$,  the number of iterations required by the gradient method to obtain $p_N\in M$ such that $\|\grad f(p_{N})\|<\epsilon$ is bounded by  ${\cal O} ( L(f(p_0) -f^*)/\epsilon^2)$.
\end{theorem}
\begin{proof} It follows from Corollary~\ref{cr:cl} and  formula \eqref{eq:GradMethod} with $t_{k}=1/L$, for all $k$,  that
\begin{equation*} 
 \frac{1}{2L}\left\|\grad f(p_k)\right\|^{2}\leq  f(p_k)-f( p_{k+1}), \qquad k=0, 1, \ldots.
\end{equation*}
By summing both sides of the above inequality  for $k=0, 1, \ldots, N$ and taking into account that $f^*\leq f(p_k)$, for all $k$, we obtain 
$$
\sum_{k=0}^{N}\left\|\grad f(p_k)\right\|^{2}\leq 2L (f(p_0)-f^*).
$$
Hence,   we have $(N+1)(\min\{\|\grad f(p_{k})\|~;~ k=0, 1,\ldots, N\})^2\leq 2L (f(p_0)-f^*)$, which proves the first statement of the theorem.  The second  statement of the theorem  is an immediate consequence of the first one.
\end{proof}
Note that the gradient method can be stated equivalently as follows: Given $p_0\in M$ define
\begin{equation} \label{eq:VGradMethod}
\displaystyle p_{k+1}=\mbox{exp} _{p_{k}}v_k, \qquad  v_{k}= \mbox{argmin}_{v\in T_{p_k}M}\left\{f(p_k) + \langle\grad f(p_k), v\rangle+\frac{1}{2t_k}\left\|v\right\|^{2}\right\}, \qquad k=0, 1, \ldots.
\end{equation}
The above alternative  formulation to the gradient method will be crucial for the   iteration-complexity analysis of the method. In particular, under  convexity of the objective function  and  non-negativity of the curvature of the Riemannian manifold,   it allows us  to  show that the rate of convergence obtained in Theorem \ref{th:grad1} can be considerably improved.  We start by showing   that the sequence of function values $\{f(p_k)\}$ converges to the optimal function value $f^*$ at a rate of convergence that is no worse than  $\mathcal{O} (1/k)$.

 \begin{theorem}\label{th:grad2}  Assume that  $M$ has non-negative curvature and $f$ is convex.  Let  $\{p_k\}$ be the  sequence  generated by the gradient method with constant stepsizes $t_{k}=1/L$, for all $k\geq 0$. Then,   for every $N\in \mathbb{N}$,  there holds
$$
 f(p_N)-f^*\leq \frac{L ~d^{2}(p_*, p_{0})}{2N}.
$$ 
As a consequence, given  a tolerance $\epsilon>0$,  the number of iterations required by the gradient method to obtain $p_N\in M$ such that $ f(p_N)-f^*< \epsilon$, is bounded by $\mathcal{O} ( [ L d^{2}(p_*, p_{0})]/\epsilon)$.
\end{theorem}
\begin{proof} 
In order to  simplify the notation, let  us define the quadratic function 
\begin{equation} \label{eq:qf} 
\phi_{j}(v):=f(p_j) + \langle\grad f(p_j), v\rangle+\frac{L}{2}\left\|v\right\|^{2}, \qquad v\in T_pM.
\end{equation}
Since $t_{k}=1/L$ for all $k\geq 0$,  using \eqref{eq:qf}, the equality  \eqref{eq:VGradMethod}  becomes
\begin{equation} \label{eq;uqf}
p_{k+1}=\mbox{exp} _{p_{k}}v_k, \qquad \quad v_{k}= \mbox{argmin}_{v\in T_{p_k}M} \phi_{k}(v) , \qquad k=0,1,\ldots.
\end{equation}
For every $k\geq 1$, let  $v_{k-1}^*\in T_{p_{k-1}}M$  be  such that  $p^*=\exp_{p_{k-1}}v_{k-1}^*$.  From \eqref{eq;uqf}, we easily see that  
$$
\phi_{k-1}( v_{k-1}^*)= \phi_{k-1}( v_{k-1}) + \frac{L}{2} \left\|v_{k-1}^*-v_{k-1}\right\|^2, \qquad k= 1, 2,\ldots.
$$
Using Lemma~\ref{le:lc} and \eqref{eq;uqf}, we have  $\phi_{k-1}(v_{k-1})\geq f(\exp_{p_{k-1}}v_{k-1})=f(p_k). $ Thus, last equality  gives 
$$
\phi_{k-1}(v_{k-1}^*)\geq  f(p_k)+ \frac{L}{2} \left\|v_{k-1}^*-v_{k-1}\right\|^2, \qquad k=1, 2,\ldots.
$$
On the other hand,  since  $f$ is convex,  the combination of    \eqref{eq:qf} with   \eqref{eq:CharDifConvexFunc}  and  taking into account  that $p^*=\exp_{p_{k-1}}v_{k-1}^*$, for all $k=1, 2, \ldots$,   we obtain  
$$
\phi_{k-1}( v_{k-1}^*)= f(p_{k-1}) + \langle\grad f(p_{k-1}), v_{k-1}^*\rangle+\frac{L}{2}\left\|v_{k-1}^*\right\|^{2}\leq  f(p^*)+\frac{L}{2}\left\|v_{k-1}^*\right\|^{2}.
$$
Hence,  using that $f^*=f(p^*) $, after some simple algebraic manipulations, the  latter two inequalities imply  that 
$$
f(p_k)-f^*\leq  \frac{L}{2}\left[ \left\|v_{k-1}^*\right\|^{2}- \left\|v_{k-1}^*-v_{k-1}\right\|^2\right],  \qquad k=1, 2, \ldots.
$$
Since the curvature of $M$  is non-negative,  the definitions of the vector $v_{k-1}^*$ and $v_{k-1}$ together with item (i) of Proposition~\ref{pr:lawcosines} imply that $d(p_*, p_k)\leq \left\|v_{k-1}^*-v_{k-1}\right\| $, for all $k=1,2, \ldots$. Thus, taking into account that  $\|v_{k-1}^*\| = d(p_*, p_{k-1})$,   we conclude from the last inequality that 
$$
f(p_k)-f^*\leq \frac{L}{2}\left[ d^{2}(p_*, p_{k-1})-d^{2}(p_*, p_k) \right], \qquad k=1,2, \ldots.
$$
Note that \eqref{eq:fvd} implies that $f( p_{k+1})\leq f(p_k)$, for  $k=0, 1, \ldots$. Hence, summing both sides of the above inequality  for $k= 1, \ldots, N$, we obtain 
$$
N[f(p_N)-f^*]\leq \frac{L}{2}[d^{2}(p_*, p_{0})-d^{2}(p_*, p_N)]\leq  \frac{Ld^{2}(p_*, p_{0})}{2},
$$
which is equivalent to  the inequality in the first statement of the theorem. The second  statement of the theorem  is an immediate consequence of the first one.
\end{proof}

\begin{corollary}  Assume that  $M$ has non-negative curvature and $f$ is convex.  Let  $\{p_k\}$ be the  sequence  generated by the gradient method with constant stepsizes $t_{k}=1/L$, for all $k\geq 0$. Then,   for every $N\in \mathbb{N}$,  there holds
$$
\min\{\|\grad f(p_{k})\|~;~ k=0, 1,\ldots, N\}\leq \frac{\sqrt{8} L~d(p_*, p_{0})}{N}.
$$ 
As a consequence, given  a tolerance $\epsilon>0$,  the number of iterations required by the gradient method to obtain $p_N\in M$ such that $ \|\grad f(p_N)\|< \epsilon$, is bounded by $\mathcal{O} ( [ L d(p_*, p_{0})]/\epsilon)$.
\end{corollary}
\begin{proof}
Using  \eqref{eq:GradMethod} with $t_{k}=1/L$, for $k=0, 1, \ldots$,  Corollary~\ref{cr:cl}  implies that
\begin{equation} \label{eq:fvd}
 \frac{1}{2L}\left\|\grad f(p_k)\right\|^{2} \leq  f(p_k)-f( p_{k+1}), \qquad k=0, 1, \ldots.
\end{equation}
On the other hand, it follows from Theorem~\ref{th:grad2} that, for evey $N\in \mathbb{N}$, we have
$$
f(p_{N+1})-f^*+ \sum_{j=\lceil N/2\rceil }^{N}\left[ f(p_{j})- f(p_{j+1})  \right]= f(p_{\lceil N/2\rceil })-f^*\leq  \frac{L ~ d^{2}(p_*, p_{0})}{\lceil N/2\rceil}\leq  \frac{2 L ~ d^{2}(p_*, p_{0})}{N}.
$$
Combining \eqref{eq:fvd} with the last inequality and taking into account  that  $f^*\leq f(p_k)$, for all $k$,   we obtain 
$$
 \frac{1}{2L}\sum_{j=\lceil N/2\rceil }^{N}\left\|\grad f(p_j)\right\|^{2} \leq f(p_{\lceil N/2\rceil})-f^*\leq  \frac{2L ~ d^{2}(p_*, p_{0})}{N}.
$$
Hence, we have   $\lceil N/2 \rceil(\min\{\|\grad f(p_{k})\|~;~ k=\lceil N/2 \rceil , \ldots, N\})^2\leq 4L^2d^{2}(p_*, p_{0})/N$, which implies the  desired inequality.  The second  statement of the corollary follows as  an immediate consequence of the first one.
\end{proof}
\subsection{Subgradient method}\label{sec:subgradient}

In this subsection, we recall the subgradient method for minimizing non-smooth convex functions on Riemannian manifolds with nonnegative curvature  and present some iteration-complexity bounds related to the method.  

\noindent
In the following,  we formally state the subgradient method  to solve \eqref{eq:OptP},  where the objective function is assumed to be convex.
\\[3mm]
\noindent
\fbox{
\begin{minipage}[h]{5.55 in}
{\bf Subgradient method}
\begin{itemize}
\item[(0)] Let an initial point $p_0 \in M$, and set $k=0$;
\item[(1)] choose a stepsize $t_k>0$, let  $s_{k}\in \partial f(p_k)$ and com putes
\begin{equation} \label{eq:SubGradMethod}
p_{k+1}:=\mbox{exp} _{p_{k}}\left(-t_{k}\,s_{k}\right);
\end{equation}
\item[(2)] set $k\leftarrow k+1$ and go to step 1.
\end{itemize}
\noindent
\end{minipage}
}
\\[3mm]
\noindent
This method is a natural extension of the well known subgradient method in the Euclidean setting. It was first proposed and analyzed in the Riemannian context in \cite{FerreiraOliveira1998}; It has been studied in different context; see, for instance, \cite{BentoMelo2012, BentoCruz Neto2013, WangLiYao2015, WangChongWangYao2015, GrohsHosseini2016}. It is worth mentioning that the subgradient method for non-smooth problems does not share the decreasing property (Corollary~\ref{cr:cl} and \eqref{eq:GradMethod}) of the gradient method. Thus, this makes its iteration-complexity analysis considerably different from the one presented in the last subsection for the gradient method. Moreover, Armijo line search is not an option for the choice of the stepsizes $t_k$. In the following, we consider the two main stepsizes rules used for the subgradient method, namely, the exogenous and the Polyak rules. The former one, does not take into account any information about  of the sequence generated by the method, while the latter one assumes the knowledge of the optimum value of the problem. Apart from these well known drawbacks, the understanding of the convergence property of the subgradient method is  fundamental for obtaining more sophisticated method to deal with non-smooth problems.

\noindent
In the next result, we recall a fundamental inequality related to the subgradient method which is essential to overcome the lack of the decreasing property of the functional values and to motivate the Polyak stepsize rule.
\begin{lemma} \label{L:ineqf}
 Let $\{p_k\}$ be the sequence generated by the subgradient method and let $p\in M$. Then,  the following inequality holds
\begin{equation*} \label{E:inequalf}
d^2(p_{k+1}, p)\leqslant d^2(p_k,p) +t_k^2\|s_{k}\|^2 + 2 t_k[f(p)-f(p_{k})],   \qquad k=0,1,\ldots.
\end{equation*}
\end{lemma}
\begin{proof}
Let $\gamma _{1} $ be the minimal   geodesic segment  joining   $p_k$ to $p$  with   $\gamma_{1}(0)=p_k$.  Note that letting  $v= \gamma _{1}'(0)$ we have $\gamma _{1} (t)= \mbox{exp} _{p_{k}}\left(t\,v\right)$, for $t\in [0, 1]$.  For  $s_{k}\in \partial f(p_k)$  define $\gamma _{2} (t)=\mbox{exp} _{p_{k}}\left(-t\,s_{k}\right)$ for $t\in [0, t_k]$. Note that  $\gamma _{2}(0)=p_k$ and from \eqref{eq:SubGradMethod} we obtain  $\gamma _{2}(t_k)=p_{k+1}$. Let $\gamma _{3} $ be the minimal   geodesic segments  joining   $p_{k+1}$ to $p$. The definitions of the geodesics segments   $\gamma _{1} $, $\gamma _{2} $ and $\gamma _{3} $ give
$$
\ell(\gamma _{1})=d(p_k, p), \quad \ell(\gamma _{2})=\|t_ks_k\|,  \quad  \ell(\gamma _{3})=d(p_{k=1}, p), \qquad  <\!\!\!\!)( \gamma _{1}^{\prime}(0),\gamma _{2}^{\prime}(0)) = <\!\!\!\!)(v, - s_k), 
$$
where $<\!\!\!\!)(u, w)$ denotes the angle between $u$ and $w$. Using item i of  Proposition~\ref{pr:lawcosines} we have
$$
d^2(p_{k+1}, p)\leqslant d^2(p_k,p) +t_k^2\|s_{k}\|^2- 2d(p_k,p) t_k\|s_{k}\|  \cos \alpha, 
$$
where $ \alpha= <\!\!\!\!)( s_k,v)$. Since $\|v\|=\ell(\gamma _{1})=d(p_k, p)$ and $ \cos \alpha=\langle -s_k,v \rangle/\|s_k\|\|v\| $, last inequality becomes 
$$
d^2(p_{k+1}, p)\leqslant d^2(p_k,p) +t_k^2\|s_{k}\|^2 +2t_k\langle s_k,v \rangle. 
$$
Due to $f$ be convex  and $p=\mbox{exp} _{p_{k}}\left(v\right)$,  the definition of subgradient in \eqref{eq:defSugrad} implies  $f(p) \geq f(p_k) + \left\langle s_k, \, v \right\rangle$, which combined with last inequality yields the desired inequality. 
\end{proof}
\noindent
The next result presents an iteration-complexity bound for the subgradient method with an exogenous stepsize rule.
\begin{theorem} 
Let $f:M \to \mathbb{R}$ be a convex   function and   Lipschitz continuous with  constant $\tau \geq 0$. Let  $\{p_k\}$ be the  sequence generated by the subgradient method with $t_{k}=\alpha_k/\|s_k\|$, for $k=0, 1, \ldots$.  Then,   for every $N\in \mathbb{N}$,  the following inequality holds 
$$
\min\left\{f(p_k) - f^*~:~ k=0, 1,\ldots, N\right\}\leq \tau \frac{d^2(p_0, p^*) + \sum_{k=0}^{N} \alpha_k^2}{2\sum_{k=0}^{N} \alpha_k}.
$$
\end {theorem}
\begin{proof}
Applying   Lemma~\ref{L:ineqf} with $p=p^*$, $t_{k}=\alpha_k/\|s_k\|$ and  using the notation $f^*=f(p^*) $ we obtain 
$$
d^2(p_{k+1}, p^*)\leqslant d^2(p_k, p^*) +\alpha_k^2 + 2 \frac{\alpha_k}{\|s_k\|}[f^*-f(p_{k})],  \qquad  s_{k}\in \partial f(p_k),  \qquad k=0,1,\ldots.
$$
Hence, performing the sum of  the above inequality  for $k=0,1, \ldots, N$, we obtain after some algebras that 
$$
 2\sum_{k=0}^{N}   \frac{\alpha_k}{\|s_k\|}[f(p_{k})-f^*] \leq d^2(p_0, p^*)- d^2(p_{N+1}, p^*) + \sum_{k=0}^{N} \alpha_k^2.
$$
Since $f$ is  Lipschitz continuous with  constant $\tau \geq 0$, we have $\|s_{k}\|\leq \tau$,   for all  $s_{k}\in \partial f(p_k)$. Therefore, 
$$
  \frac{2}{\tau} \min\left\{f(p_k) - f^*~:~ k=0, 1,\ldots, N\right\}\sum_{k=0}^{N}   \alpha_k \leq d^2(p_0, p^*) + \sum_{k=0}^{N} \alpha_k^2, 
$$
which is 	equivalent to the desired inequality.
\end{proof}
\noindent
The next result presents an iteration-complexity bound for the subgradient method with Polyak stepsize rule.
\begin{theorem} 
Let $f:M \to \mathbb{R}$ be a convex   function and   Lipschitz continuous with  constant $\tau \geq 0$. Let   $\{p_k\}$ be the  sequence  generated by the subgradient method  with  $t_{k}= [f(p_k) - f^*]/ \|s_k\|^2$,   for all  $k=0, 1, \ldots$.
Then,   for every $N\in \mathbb{N}$,  there  holds  
$$
 \sum_{k=0}^{N} [f(p_{k})-f^*]^2\leq \tau^2d^2(p_0, p^*).
$$
As a consequence, $\min\left\{f(p_k) - f^*~:~ k=0, 1,\ldots, N\right\}\leq [\tau d(p_0, p^*)]/\sqrt{N+1}$.
\end {theorem}
\begin{proof}
Applying   Lemma~\ref{L:ineqf} with $p=p^*$,  $t_{k}= [f(p_k) - f^*]/ \|s_k\|^2$ and  using the notation $f^*=f(p^*) $ we obtain
$$
\frac{[f(p_k) - f^*]^2}{\|s_k\|^2}\leqslant d^2(p_k,p^*) - d^2(p_{k+1}, p^*), \qquad k=0, 1,\ldots.
$$
Performing the sum of  the above inequality  for $k=0,1, \ldots, N$, we conclude    that 
$$
 \sum_{k=0}^{N}   \frac{[f(p_{k})-f^*]^2}{\|s_k\|^2}\leq d^2(p_0, p^*).
$$
Since $f$ is  Lipschitz continuous with  constant $\tau \geq 0$, we have $\|s_{k}\|\leq \tau$,   for all $k\geq 0$. Therefore, the first statement of the theorem follows from the last inequality.
 The second statement of the theorem is an immediate consequence of the first one.
\end{proof}
\subsection{Proximal point method}\label{sec:proximalpoint}
In this subsection, we recall the proximal point method on a Hadamard manifold and  present two results. The first one  shows an important inequality which is essential to prove  the  convergence rate bound of the method obtained in our second result.

 \noindent 
 In the following,  we formally state the proximal point method to solve \eqref{eq:OptP}.
\\[3mm]
\noindent
\fbox{
\begin{minipage}[h]{5.55 in}
{\bf Proximal point method}
\begin{itemize}
\item[(0)] Let an initial point $p_0 \in M$ and $\{\lambda_k\}\subset \mathbb{R}_{++}$.  Set $k=0$;
\item[(1)] computes
\begin{equation} \label{eq:SeqProximal}
p_{k+1}=\mbox{argmin}_{p\in M}\left\{f(p)+\frac{\lambda_k}{2} d^2(p_k, p)\right\};
\end{equation}
\item[(2)] set $k\leftarrow k+1$ and go to step 1.
\end{itemize}
\noindent
\end{minipage}
}
\\[3mm]
\noindent
 The proximal method was first proposed and analyzed in the Riemannian setting in \cite{FerreiraOliveira2002}. Since then, it has been the subject of intense research; see, for example,  \cite{LiLopesMartin-Marquez2009, BentoFO2010, BentoFerreiraOliveira2015, BentoNetoOliveira2016} and reference therein. As far as we know, all the papers studying  convergence of the proximal point method above analyze only its asymptotic convergence property. Next, we discuss a basic result which will be essential to obtain iteration-complexity bound for the proximal point method.   

\begin{proposition} \label{pr:StrongCharact}
Let  $M$ be a Hadamard manifold,   $f:M\to\mathbb{R}$  be a  convex function,  $ {\bar p}\in M$ and $\mu >0$.
Then, for each   $p, q \in M$ and  $s\in \partial f(p)$ the following inequality holds
$$
f(q)+ \frac{\mu}{2} d^2(q, {\bar p}) \geq  f(p)+ \frac{\mu}{2} d^2(p, {\bar p}) +\left\langle s - \mu \exp^{-1}_{p}{\bar p}, ~  \exp^{-1}_{ p}q \right\rangle  + \frac{\mu}{2}d^2(q, p).
$$ 
\end{proposition}
\begin{proof}
Let  $p, q  \in M$.  Due to  $f$ be convex, we can  take   $v= \exp^{-1}_{ p}q$ into inequality  \eqref{eq:CharConvexFunc} to  obtain 
\begin{equation} \label{eq:sih}
f(q) \geq f(p) + \left\langle s, \, \exp^{-1}_{ p}q\right\rangle, \qquad s\in \partial f(p).
\end{equation}
On the other hand,  since   $M$ is  a Hadamard manifold, it follows from  Proposition~\ref{pr:lawcosines}(ii)   that
$$
d^2(q,{\bar p})\geq d^2(q,p)+d^2(p,{\bar p})-2\left\langle \exp_{p}^{-1}{\bar p}, ~ \exp_{p}^{-1}q\right\rangle.
$$
Multiplying the last inequality   by $\mu/2$ and  summing the result  with   \eqref{eq:sih},  the desired  inequality follows.
\end{proof}
Next theorem presents our main result related to the convergence rate of the proximal point method.
\begin{theorem} Let  $M$ be a Hadamard manifold and   $f:M\to\mathbb{R}$  be  a convex function. Let  $\{p_k\}$ be the  sequence  generated by the proximal point method with $\lambda\geq \lambda_k>0$, for $k=0, 1, \ldots$. Then,   for every $N\in \mathbb{N}$,  there holds
$$
 f(p_N)-f^*\leq \frac{ \lambda d^{2}(p_*, p_{0})}{2 [N+1]}.
 $$
 As a consequence, given  a tolerance $\epsilon>0$,  the number of iterations required by the proximal point method to obtain $p_N\in M$ such that $ f(p_N)-f^*\leq \epsilon$, is bounded by $\mathcal{O} (  \lambda d^{2}(p_*, p_{0})/ \epsilon)$.
\end{theorem}
\begin{proof}
Since $\{p_k\}$ is  the  sequence  defined in \eqref{eq:SeqProximal},  we have
$$
0\in   \partial f(p_{k+1})-\lambda_k \exp^{-1}_{p_{k+1}}{p_k},  \qquad k=0, 1, \ldots.
$$
Applying Proposition~\ref{pr:StrongCharact} with ${\bar p}= p_k$,  $p=p_{k+1}$ and $\mu=\lambda_k$,  and  considering the last inclusion  we obtain, for every $q\in M$, that 
\begin{equation} \label{eq:dc1}
f(q)+ \frac{\lambda_k}{2} d^2(q, p_k) \geq  f(p_{k+1})+ \frac{\lambda_k}{2} d^2(p_{k+1},p_k)  + \frac{\lambda_k}{2}d^2(q, p_{k+1}), \qquad \;k=0, 1, \ldots.
\end{equation}
It follows by taking  $q=p^*$  in the  last inequality  and  using $f^*=f(p^*) $ that
$$
0\leq f (p_{k+1})-f^* \leq   \frac{\lambda_k}{2} \left[ d^2(p^*, p_k) - d^2(p_*, p_{k+1})\right],  \qquad k=0, 1, \ldots.
$$
Hence, summing both sides of the last inequality  for $k=0,1, \ldots, N$ and using $\lambda\geq \lambda_k$, we obtain
\begin{equation} \label{eq:feq1}
\sum_{k=0}^{N}[ f (p_{k+1})-f^*] \leq   \frac{\lambda}{2} \left[ d^2(p_0, p^*) - d^2(p_*, p_{N})\right] \leq   \frac{\lambda}{2}  d^2(p_0, p^*).
\end{equation}
Letting  $q=p_{k}$  in \eqref{eq:dc1} we conclude that $f(p_k) \geq    f(p_{k+1})$, for  all $k=0, 1, \ldots$. Therefore,  \eqref{eq:feq1} implies that
$
[N+1][ f (p_{N})-f^*] \leq     \lambda d^2(p_0, p^*)/2, 
$
which proves the first statement of the theorem. The last statement of the theorem is an immediate consequence of the first one.
\end{proof}
\section{Final remarks}
In this paper,  we analyze iteration-complexity of gradient, subgradient and proximal point methods.  We expect that  this paper will contribute to the development of the iteration-complexity studies of optimization methods in the Riemannian setting. It remains an open and challenging problem to show whether or not accelerated  schemes (see, \cite{BeckTeboulle2009, Nesterov2004}) can be extended to handle convex optimization problems in the  Riemannian setting.  Finally, it would be interesting to continue the studies in this direction in order to go further and analyze stochastic  versions of the above  algorithms in a Riemannian context.  

\subsection*{acknowledgements}
The work  was supported  by  CNPq Grants  458479/2014-4,  312077/2014-9,  305158/2014-7.


\def\cprime{$'$}

\end{document}